\documentclass[12pt]{amsart}
\usepackage{amsmath, amsthm, amscd, amsfonts}

\newtheorem{theorem}{Theorem}[section]
\newtheorem{lemma}[theorem]{Lemma}

\newtheorem{corollary}[theorem]{Corollary}
\theoremstyle{definition}

\theoremstyle{remark}

\numberwithin{equation}{section}

\newfont{\kh}{msbm10}

\begin{document}
\title[The reverse order law for Moore-Penrose inverses]
{The reverse order law for Moore-Penrose inverses of operators on Hilbert C*-modules}

\author{K. Sharifi}
\address{Kamran Sharifi, \newline Department of Mathematics,
Shahrood University of Technology, P. O. Box 3619995161-316,
Shahrood, Iran.} \email{sharifi.kamran@gmail.com}
\author{B. Ahmadi Bonakdar}
\address{Behnaz Ahmadi Bonakdar, \newline Department of Mathematics,
International Campus of Ferdowsi University, Mashhad, Iran.}
\email{b.ahmadibonakdar@gmail.com}


\begin{abstract}
Suppose
$T$ and $S$ are bounded adjointable operators between Hilbert C*-modules admitting
bounded Moore-Penrose inverse operators.
Some necessary and sufficient conditions are given for the
reverse order law $(TS)^{ \dag} =S^{ \dag} T^{ \dag}$ to hold.
In particular, we show that the equality holds
if and only if $Ran(T^{\ast}TS) \subseteq Ran(S)$ and
$Ran(SS^{\ast}T^{\ast}) \subseteq Ran(T^{\ast}),$ which was studied first by Greville
[{\it SIAM Rev. 8 (1966) 518--521}] for matrices.
\end{abstract}

 \maketitle

\section{Introduction and preliminaries.}
It is well-known that for invertible operators (or nonsingular
matrices) $T, S$ and $TS$, $(TS)^{-1}=S^{-1} T^{-1}$. However,
this so-called reverse order law is not necessarily true for other
kind of generalized inverses. An interesting problem is, for given
operators (or matrices) $TS$ with $TS$ meaningful, then under what
conditions, $(TS)^{ \dag}=S^{ \dag} T^{ \dag}$? The problem first
studied by Greville \cite{GRE} and then reconsidered by Bouldin and
Izumino \cite{BouldinLaw1, IzuminoLaw1}. Many authors discussed
the problem like this, see e.g. \cite{orderlaw0, orderlaw01, orderlaw1, KDC, orderlaw2}
and references therein. An special case, when $S=T^*$, was
given by Moslehian et~al. \cite{SHA/Gram} for a Moore-Penrose
invertible operator $T$ on Hilbert C*-modules. The later paper
and the work of \cite{orderlaw1, GRE} motivate us to study the problem
in the framework of Hilbert C*-modules.

The notion of a Hilbert C*-module is a generalization of the notion
of a Hilbert space. However, some well known properties
of Hilbert spaces like Pythagoras' equality, self-duality, and
even decomposition into orthogonal complements do not hold in the
framework Hilbert modules. The first use of such objects was made by
I. Kaplansky \cite{op5} and then studied more in the work of
W. L. Paschke \cite{op8}. Let us quickly recall the
definition of a Hilbert C*-module.

Suppose that $ \mathcal{A} $ is an
arbitrary C*-algebra and $E$ is a linear space
which is a right $ \mathcal{A}$-module and the scalar multiplication
satisfies $ \lambda (xa)=x(\lambda a)=(\lambda x)a $ for all $ x\in E, ~a\in
\mathcal{A}, \lambda \in \mathbb{C}$. The $ \mathcal{A}$-module $E$ is called a
{\it pre-Hilbert $ \mathcal{A}$-module} if there exists an $\mathcal{A}$-valued
map $ \langle .,.\rangle: E \times E \to \mathcal{A} $ with the
following properties:
\newcounter{cou001}
\begin{list}{(\roman{cou001})}{\usecounter{cou001}}
\item $ \langle x, y+\lambda z\rangle =\langle x,y\rangle +\lambda \langle
 x,z\rangle $;  for all $ x,y,z\in E ,\lambda \in \mathbb{C},$
\item $ \langle x,ya\rangle=\langle x,y \rangle a;$  for all $x,y\in E $ and  $ a\in A$,
\item $ \langle x,y\rangle ^{\ast}=\langle y,x \rangle;$ for all $x,y\in E $,
\item $ \langle x,x\rangle \geq 0  $ and $ \langle x,x\rangle=0 $ if and only if $ x=0.$
\end{list}
The $ \mathcal{A}$-module  $ E $ is called a {\it Hilbert C*-module}
if $E$ is complete with respect to the norm $ \Vert x\Vert =\Vert \langle
 x,x\rangle \Vert ^{1/2}.$ For any pair of Hilbert C*-modules $ E_{1}$ and
$ E_{2}$, we  define
$ E_{1}\oplus E_{2}= \{  (e_1,e_2) | ~ e_{1}\in E_{1} ~{\rm and} ~e_{2}\in E_{2} \}$
which is also a Hilbert C*-module whose $ \mathcal{A} $-valued inner product
is given by $$ \langle (x_1,y_1), (x_2,y_2) \rangle =\langle x_{1},x_{2} \rangle +\langle y_{1},y_{2}
 \rangle, ~~{\rm  for}~ x_1, x_2 \in E_1~ {\rm and} ~y_1, y_2 \in E_2.$$

If $V$ is a (possibly non-closed) $\mathcal A$-submodule
of $E$, then $V^\bot :=\{ y \in E: ~ \langle x,y \rangle=0,~ \ {\rm for} \ {\rm
all}\ x \in V \} $ is a closed $\mathcal A$-submodule of $E$ and
$ \overline{V} \subseteq V^{ \perp \, \perp}$. A Hilbert $\mathcal
A$-submodule $V$ of a Hilbert $\mathcal A$-module $E$ is
orthogonally complemented if $V$ and its orthogonal complement
$V^\bot$ yield $E=V \oplus V^\bot $, in this case, $V$ and its
biorthogonal complement $V^{ \perp \, \perp}$ coincide. For the
basic theory of Hilbert C*-modules we refer to the book by
E.~C.~Lance \cite{LAN}. Note that every Hilbert space is a
Hilbert $ \mathbb{C}$-module and every C*-algebra ${\mathcal A}$
can be regarded as a Hilbert ${\mathcal A}$-module via
$\langle a, b\rangle = a^*b$ when $a, b \in
{\mathcal A}$.

Throughout this paper we assume that $\mathcal{A}$ is an arbitrary
C*-algebra. We use $[ \cdot, \cdot]$ for commutator of two elements.
The notations $Ker(\cdot)$ and $Ran(\cdot)$ stand for kernel and range of operators,
respectively. Suppose $E$ and $F$ are Hilbert $ \mathcal{A}$-modules, $\mathcal{L}(E,F)$ denotes the set
of all bounded adjointable operators from $E$ to $F$, that is, all operator
$T: E \to F$ for which there exists $ T^{\ast}:F \to E$ such that $ \langle Tx,y\rangle=
\langle x , T^{\ast} y \rangle $, for all $ x\in E$ and $ y \in F$.

Closed submodules of Hilbert modules need not to be orthogonally
complemented at all, however we have the following well known
results. Suppose $T$ in $ \mathcal{L}(E,F)$, the operator $T$ has
closed range if and only if $T^*$ has. In this case, $E=Ker(T)
\oplus Ran(T^*)$ and $F=Ker(T^*) \oplus Ran(T)$, cf.
\cite[Theorem 3.2]{LAN}. In view of \cite[Lemma
2.1]{SHA/PARTIAL}, $Ran(T)$ is closed if and only if $Ran(T
\,T^*)$ is, and in this case, $Ran(T)=Ran(T\, T^*)$.

Let $T \in \mathcal{L}(E,F)$. The Moore-–Penrose inverse $T^{ \dag}$ of $T$ (if it exists)
is an element $X \in \mathcal{L}(F,E)$ which satisfies
\begin{enumerate}
  \item $TXT=T$,
  \item $XTX=X$,
  \item $(TX)^*=TX$,
  \item $(XT)^*=XT$.
\end{enumerate}
If $ \theta \subseteq \{1, 2, 3, 4 \}$, and X satisfies the equations $(i)$ for all $i \in \theta$,
then $X$ is an $ \theta$-inverse of $T$. The set of all $ \theta$-inverses of $T$
is denoted by $T\{ \theta \}$. In particular, $T \{ 1, 2, 3, 4 \}=\{ T^{ \dag} \}$. The properties
(1) to (4) imply that $T^{ \dag}$ is unique
and $ T^{ \dag} T$ and $ T \, T^{ \dag} $ are orthogonal
projections. Moreover, $Ran( T^{ \dag} )=Ran( T^{ \dag}  T)$,
$Ran( T )=Ran( T \, T^{ \dag})$,  $Ker(T)=Ker( T^{ \dag} T)$ and
$Ker(T^{ \dag})=Ker( T \, T^{ \dag} )$ which lead us to $ E= Ker(
T^{ \dag} T) \oplus Ran( T^{ \dag} T)= Ker(T) \oplus Ran( T^{
\dag} )$ and $F= Ker(T^{ \dag}) \oplus Ran(T).$ We also have
$Ran(T^{ \dag})=Ran(T^{ *})$ and $Ker(T^ \dag)=Ker(T^*)$.

 Xu and Sheng in \cite{Xu/Sheng} have shown that a bounded
adjointable operator between two Hilbert C*-modules admits a
bounded Moore-Penrose inverse if and only if the operator has
closed range. The reader should be aware of the fact that a
bounded adjointable operator may admit an unbounded operator as
its Moore-Penrose, see \cite{FS2, Gu2, SHA/PARTIAL, SHA/Groetsch} for
more detailed information.

It is a classical result of Greville \cite{GRE}, that
$(TS)^{\dag}= S^{ \dag} T^{ \dag}$ if and only if $T^{ \dag}TSS^*T^*=SS^*T^*$
and $SS^{ \dag}T^*TS=T^*TS$ (or equivalently, $Ran(SS^*T^*) \subseteq Ran(T^*)$ and
$Ran(T^*TS) \subseteq Ran(S)$) for Moore-Penrose invertible matrices $T$ and $S$. The
present paper is an extension of some results of \cite{
orderlaw1, GRE, SHA/Gram} to Hilbert C*-modules settings. Indeed,  we give some
necessary and sufficient conditions for reverse order law for the
Moore-–Penrose inverse by using the matrix form of bounded adjointable module
maps. These enable us to derive  Greville's result for bounded adjointable module
maps.

The matrix form of a bounded adjointable operator $ T \in \mathcal{L}(E,F) $ is induced
by some natural decompositions of Hilbert C*-modules. If
$F=M \oplus M^{\perp} , E=K
\oplus K^{\perp} $ then $T$ can be written as the following $
2\times 2 $ matrix
\begin{equation} \label{lawMatrix}
T= \left[\begin{array}{cc}
T_{1} & T_{2} \\
T_{3} & T_{4} \\
\end{array}\right]
\end{equation}
with operator entries, $ T_{1}\in \mathcal{L}(K,M) , T_{2}\in
\mathcal{L}(K^{\perp},M), T_{3}\in \mathcal{L}(K,M^{\perp}) $
 and $ T_{4}\in \mathcal{L}(K^{\perp},M^{\perp}) $.

\begin{lemma} \label{ahmadi1}
Let $ T \in \mathcal{L}(E,F) $ have a closed range. Then $T$ has
the following matrix decomposition with respect to the orthogonal
decompositions of submodules $ E=Ran(T^{\ast})\oplus Ker(T) $ and $
F=Ran(T)\oplus Ker(T^{\ast})$:
\begin{equation*} \label{law0}
T=\left[\begin{array}{cc}
T_{1} & 0 \\
0 & 0 \\
\end{array}\right]: \left[\begin{array}{c}
Ran(T^{\ast}) \\
Ker(T) \\
\end{array}\right] \to\left[\begin{array}{c}
Ran(T) \\
Ker(T^{\ast}) \\
\end{array}\right],
\end{equation*}
where $ T_{1} $ is invertible. Moreover,
\begin{equation*} \label{law00}
T^{\dagger}=\left[\begin{array}{cc}
T_{1}^{-1} & 0 \\
0 & 0 \\
\end{array}\right]: \left[\begin{array}{c}
Ran(T) \\
Ker(T^{\ast}) \\
\end{array}\right] \to\left[\begin{array}{c}
Ran(T^{\ast}) \\
Ker(T) \\
\end{array}\right].
\end{equation*}
\end{lemma}
\begin{proof}
The operator $T$ and its adjoint $T^*$ have the following representations:
$$T= \left[\begin{array}{cc}
T_{1} & T_{2} \\
T_{3} & T_{4} \\
\end{array}\right]: \left[\begin{array}{c}
Ran(T^*) \\
Ker(T) \\
\end{array}\right] \to\left[\begin{array}{c}
Ran(T) \\
Ker(T^{\ast}) \\
\end{array}\right], $$
$$ T^{\ast}= \left[\begin{array}{cc}
T_{1}^{\ast} & T_{3}^{\ast} \\
T_{2}^{\ast} & T_{4}^{\ast} \\
\end{array}\right]: \left[\begin{array}{c}
Ran(T) \\
Ker(T^*) \\
\end{array}\right] \to~\left[\begin{array}{c}
Ran(T^{\ast}) \\
Ker(T) \\
\end{array}\right].$$
From $ T^{\ast}(Ker(T^{\ast}))=\lbrace 0\rbrace $ we obtain
 $ T^{\ast}_{3}=0 $ and $ T^{\ast}_{4}=0 $, so $ T_{3}=0 $
and $ T_{4}=0 $. Since $T(Ker(T))=\{ 0 \}$, $T_{2}=0 $ and so
$T=\left[ \begin{smallmatrix} T_{1} & 0 \\
0 & 0  \end{smallmatrix} \right]$.

Since $Ran(T)$ is close, $ T_{1} $ possesses a bounded adjointable inverse from
$Ran(T)$ onto $Ran(T^{\ast})$. Now, it is easy to check that the matrix
$\left[ \begin{smallmatrix} T_{1}^{-1} & 0 \\ 0 & 0  \end{smallmatrix} \right]$
is the Moore--Penrose inverse of $T=\left[ \begin{smallmatrix} T_{1} & 0 \\
0 & 0  \end{smallmatrix} \right]$.
\end{proof}

\begin{lemma}\label{ahmadi2}
let $ T\in \mathcal{L}(E,F) $ have a closed range. Let $E_1, E_2$ be closed submodules of
$E$ and $F_1, F_2$ be closed submodules of $F$ such that $ E=E_{1}\oplus E_{2}$ and
$ F=F_{1}\oplus F_{2}$. Then the operator $T$ has
 the following matrix representations with respect to the orthogonal sums
of submodules $E=Ran(T^*) \oplus Ker(T)$ and $F=Ran(T) \oplus Ker(T^*)$:

\begin{equation} \label{law1}
T=\left[\begin{array}{cc}
T_{1} & T_{2} \\
0 & 0 \\
\end{array}\right]~:\left[\begin{array}{c}
E_{1} \\
E_{2} \\
\end{array}\right] ~ \to~ \left[\begin{array}{c}
Ran(T) \\
Ker(T^{\ast}) \\
\end{array}\right],
\end{equation}
where $ D=T_{1}T_{1}^{\ast}+T_{2}T_{2}^{\ast} \in \mathcal{L}(Ran(T))$ is positive
and invertible. Moreover,
\begin{equation} \label{law2} T^{\dagger}=\left[\begin{array}{cc}
T_{1}^{\ast}D^{-1} & 0  \\
T_{2}^{\ast}D^{-1} & 0 \\
\end{array}\right].
\end{equation}

\begin{equation} \label{law3}
 T=\left[\begin{array}{cc}
T_{1} & 0 \\
T_{2} & 0 \\
\end{array}\right] : \left[\begin{array}{c}
Ran(T^{\ast}) \\
Ker(T) \\
\end{array}\right] \to~\left[\begin{array}{c}
F_{1} \\
F_{2} \\
\end{array}\right],
\end{equation}
where $\mathfrak{D}= T_{1}^{\ast}T_{1}+T_{2}^{\ast}T_{2} \in \mathcal{L}(Ran(T^*))$ is positive
and invertible. Moreover,
\begin{equation} \label{law4}
T^{\dagger}= \left[\begin{array}{cc}
\mathfrak{D}^{-1}T_{1}^{\ast} & \mathfrak{D}^{-1}T_{2}^{\ast} \\
0 & 0 \\
\end{array}\right].
\end{equation}
\end{lemma}

\begin{proof}
We prove only the matrix representations (\ref{law1}) and (\ref{law2}), the
proof of (\ref{law3}) and (\ref{law4}) are analogous.
The operator $T$ has the following representation:
\begin{equation*} \label{law5}
T=\left[\begin{array}{cc}
T_{1} & T_{2} \\
T_{3} & T_{4} \\
\end{array}\right] :\left[\begin{array}{c}
E_{1} \\
E_{2} \\
\end{array}\right] \to~\left[\begin{array}{c}
Ran(T) \\
Ker(T^{\ast}) \\
\end{array}\right],
\end{equation*}
which yields
\begin{equation*} \label{law6}
T^{\ast}=\left[\begin{array}{cc}
T_{1}^{\ast} & T_{3}^{\ast} \\
T_{2}^{\ast} & T_{4}^{\ast} \\
\end{array}\right] ~:\left[\begin{array}{c}
Ran(T) \\
Ker(T^{\ast}) \\
\end{array}\right]\to~\left[\begin{array}{c}
E_{1} \\
E_{2} \\
\end{array}\right].
\end{equation*}
From $ T^{\ast}(Ker(T^{\ast}))=\lbrace 0\rbrace $ we obtain $
T^{\ast}_{3}=0 $
and $ T^{\ast}_{4}=0 $. Then $ T_{3}=0 $ and $ T_{4}=0 $ which yield
the matrix form (\ref{law1}) of $T$. Consequently, the adjoint operator
$T^*$ has the matrix representation
\begin{equation*} \label{law66}
T^{\ast}=\left[\begin{array}{cc}
T_{1}^{\ast} & 0 \\
T_{2}^{\ast} & 0 \\
\end{array}\right] ~:\left[\begin{array}{c}
Ran(T) \\
Ker(T^{\ast}) \\
\end{array}\right]\to~\left[\begin{array}{c}
E_{1} \\
E_{2} \\
\end{array}\right].
\end{equation*}
We therefore have
\begin{equation} \label{law7}
 T \, T^{\ast}=\left[\begin{array}{cc}
D & 0 \\
0 & 0 \\
\end{array}\right]:\left[\begin{array}{c}
Ran(T) \\
Ker(T^*) \\
\end{array}\right]\to~\left[\begin{array}{c}
Ran(T) \\
Ker(T^{\ast}) \\
\end{array}\right].
\end{equation}
where $ D=T_{1}T_{1}^{\ast}+T_{2}T_{2}^{\ast}:Ran(T)\to~ Ran(T) $.
From $ Ker(TT^{\ast})=Ker(T^{\ast})$ it follows that $D$ is
injective. From $ Ran(TT^{\ast})=Ran(T)$ it follows that $D$ is surjective.
Hence, $D$ is invertible. Using \cite[Corollary 2.4]{SHA/Gram} and (\ref{law7}) we obtain
\begin{equation*} \label{law77}
T^{\dagger}=T^{\ast}(TT^{\ast})^{\dagger}=\left[\begin{array}{cc}
T_{1}^{\ast} & 0 \\
T_{2}^{\ast} & 0 \\
\end{array}\right] \left[\begin{array}{cc}
D^{-1} & 0 \\
0 & 0 \\
\end{array}\right]=\left[\begin{array}{cc}
T_{1}^{\ast}D^{-1} & 0 \\
T_{2}^{\ast}D^{-1} & 0 \\
\end{array}\right].
\end{equation*}
\end{proof}
\section{The reverse order law}
We begin our section with the following useful facts about
the product of module maps with closed range. Suppose $E, F$ and $G$
are Hilbert C*-modules and $ S\in \mathcal{L}(E,F) $ and
$ T \in \mathcal{L}(F,G) $ are bounded
adjointable operators with closed ranges. Then  $TS$ has
closed range if and only if $T^{\dagger}TSS^{\dagger}$ has, if and only if
$Ker(T)+Ran(S)$ is an orthogonal summand in $F$, if an only if
 $Ker(S^*)+Ran(T^*)$ is an orthogonal summand in $F$. For the proofs
of the results and historical notes about the problem we refer to
\cite{SHA/PRODUCT} and references therein.

\begin{theorem}\label{ahmadi3}
Suppose $E, F$ and $G$ are Hilbert C*-modules and $ S\in
\mathcal{L}(E,F)$, $T \in \mathcal{L}(F,G)$ and $TS \in \mathcal{L}(E,G)$
have closed ranges.
Then following statements are equivalent:
\begin{list}{(\roman{cou001})}{\usecounter{cou001}}
\item $ TS(TS)^{\dagger}=TSS^{\dagger}T^{\dagger},$
\item $T^{\ast}TS=SS^{\dagger}T^{\ast}TS,$
\item $ S^{\dagger}T^{\dagger}\in (TS) \{ 1,2,3 \}.$
\end{list}
\end{theorem}

\begin{proof}
Using Lemma \ref{ahmadi1}, the operator $S$ and its Moore-Penrose inverse $S^{ \dag}$
have the following matrix forms:
\begin{eqnarray*}
S =  \left[\begin{array}{cc}
S_{1} & 0 \\
0 & 0 \\
\end{array}\right] & : & \left[\begin{array}{c}
Ran(S^{\ast}) \\
Ker(S) \\
\end{array}\right] \to~\left[\begin{array}{c}
Ran(S) \\
Ker(S^{\ast}) \\
\end{array}\right],\\
 S^{\dagger}  =  \left[\begin{array}{cc}
S_{1}^{-1} & 0 \\
0 & 0 \\
\end{array}\right] & : & \left[\begin{array}{c}
Ran(S) \\
Ker(S^{\ast}) \\
\end{array}\right] \to~\left[\begin{array}{c}
Ran(S^{\ast}) \\
Ker(S) \\
\end{array}\right].
\end{eqnarray*}
From Lemma \ref{ahmadi2} it follows that the
operator $T$ and $T^{ \dag}$ have the following matrix forms:
\begin{eqnarray*}
T & = & \left[\begin{array}{cc}
T_{1} & T_{2} \\
0 & 0 \\
\end{array}\right] : \left[\begin{array}{c}
Ran(S) \\
Ker(S^{\ast}) \\
\end{array}\right] \to~\left[\begin{array}{c}
Ran(T) \\
Ker(T^{\ast}) \\
\end{array}\right], \\
T^{\dagger} & = & \left[\begin{array}{cc}
T_{1}^{\ast}D^{-1} & 0 \\
T_{2}^{\ast}D^{-1} & 0 \\
\end{array}\right],
\end{eqnarray*}
where $ D=T_{1}T_{1}^{\ast}+T_{2}T_{2}^{\ast}$ is invertible and positive
in $ \mathcal{L}(Ran(T)) $. Then we have the following products
$$ TS=\left[\begin{array}{cc}
T_{1}S_{1} & 0 \\
0 & 0 \\
\end{array}\right], ~~
(TS)^{\dagger}=\left[\begin{array}{cc}
(T_{1}S_{1})^{\dagger} & 0 \\
0 & 0 \\
\end{array}\right],~~~
S^{\dagger}T^{\dagger}=\left[\begin{array}{cc}
S_{1}^{-1}T_{1}^{\ast}D^{-1} & 0 \\
0 & 0 \\
\end{array}\right].$$
It is easy to check that the following three expressions
in terms of $T_1$, $T_2$ and $S_1$ are equivalent to our statements.
\begin{enumerate}
\item $T_{1}S_{1}(T_{1}S_{1})^{\dagger}=T_{1}T_{1}^{\ast}D^{-1}$, which is equivalent to (i).
\item $T_{2}^{\ast}T_{1}=0$, which is equivalent to (ii).
\item $T_{1}T_{1}^{\ast}D^{-1}T_{1}=T_{1}$ and
      $[T_{1}T_{1}^{\ast},D^{-1}]=0 $, which are equivalent to (iii).
\end{enumerate}
Note that $ [T_{1}T_{1}^{\ast}, D^{-1}]=0$, since $ T_{1}S_{1}(T_{1}S_{1})^{\dagger}$ is selfadjoint.
We show that $(3) \Rightarrow (2) \Leftrightarrow (1) \Rightarrow (3)$.

To prove $(1) \Leftrightarrow (2)$, we observe that
$ T_{1}S_{1}(T_{1}S_{1})^{\dagger}=T_{1}T_{1}^{\ast}D^{-1}$ if and only if
$(T_{1}S_{1})^{\dagger}=(T_{1}S_{1})^{\dagger}
T_{1}T_{1}^{\ast}D^{-1}.$ The last statement is obtained by multiplying the first expression by
$ (T_{1}S_{1})^{\dagger} $ from the left side, or multiplying the
second expression by $ T_{1}S_{1} $ from the left side, and using
  $ T_{1}T_{1}^{\ast}=T_{1}S_{1}S_{1}^{-1}T_{1}^{\ast}$.
We therefore have
\begin{eqnarray*}
(T_{1}S_{1})^{\dagger}=(T_{1}S_{1})^{\dagger}T_{1}T_{1}^{\ast} D^{-1}
 & \Leftrightarrow& (T_{1}S_{1})^{\dagger}(T_{1}T_{1}^{\ast}+T_{2}
                     T_{2}^{\ast})=(T_{1}S_{1})^{\dagger}T_{1}T_{1}^{\ast} \\
 & \Leftrightarrow& (T_{1}S_{1})^{\dagger}T_{2}T_{2}^{\ast}=0 \\
 & \Leftrightarrow& Ran(T_{2}T_{2}^{\ast})\subseteq Ker((T_{1}S_{1})^{\dagger})=Ker((T_{1}S_{1})^{\ast}) \\
 & \Leftrightarrow& S_{1}^{\ast}T_{1}^{\ast}T_{2}T_{2}^{*}=0 ~~\Leftrightarrow ~~ T_{2}T_{2}^{*}T_{1}=0\\
 & \Leftrightarrow& Ran(T_1) \subseteq Ker(T_{2}T_{2}^{*})=Ker (T_{2}^{*})\\
 & \Leftrightarrow& T_{2}^{*}T_1=0.
\end{eqnarray*}

To demonstrate $(1) \Rightarrow (3)$, we multiply $ T_{1}S_{1}(T_{1}S_{1})
^{\dagger}=T_{1}T_{1}^{\ast}D^{-1} $ by $ T_{1}S_{1} $ from the right side,
we find $ T_{1}T_{1}^{\ast}D^{-1}T_{1}=T_{1}$, i.e. (3) holds.

Finally, we prove $(3) \Rightarrow (2)$. If $T_{1}T_{1}^{\ast}D^{-1}T_{1}
=T_{1}$ and $ [T_{1}T_{1}^{\ast} , D^{-1}]=0 $, then $ T_{1}T_{1}^{\ast}T_{1}=DT_{1}=T_{1}T_{1}^{\ast}T_{1}+T_{2}T_{2}^{\ast}
T_{1}$. Consequently,   $T_{2}T_{2}^{\ast}T_{1}=0$ which implies
$T_{2}T_{1}^{\ast}=0 $, since $ Ran(T_{1})\subseteq Ker(T_{2}T_{2}^{\ast})=Ker(T_{2}^{\ast})$.
\end{proof}

\begin{theorem}\label{ahmadi4}
Suppose $E, F$ and $G$ are Hilbert C*-modules and $ S\in
\mathcal{L}(E,F) $, $T \in \mathcal{L}(F,G)$ and $TS \in \mathcal{L}(E,G)$
have closed ranges.
Then following statements are equivalent:
\begin{list}{(\roman{cou001})}{\usecounter{cou001}}
\item $(TS)^{\dagger}TS=S^{\dagger}T^{\dagger}TS,$
\item $TSS^{\ast}=TSS^{\ast}T^{\dagger}T,$
\item $S^{\dagger}T^{\dagger}  \in (TS)\lbrace 1,2,4 \rbrace.$
\end{list}
\end{theorem}

\begin{proof}
The operators $T$, $S$ and $TS$ and their Moore-Penrose inverses
have the same matrix representations as in the previous theorem.
To prove the assertions, we first find the equivalent expressions for our statements
in terms of $T_1$, $T_2$ and $S_1$.
\begin{enumerate}
\item $(T_{1}S_{1})^{\dagger}T_{1}S_{1}=S_{1}^{-1}T_{1}^{\ast}D^{-1}T_{1}S_{1}$,
      which is equivalent to (i).
\item $T_{1}S_{1}S_{1}^{\ast}T_{1}^{\ast}D^{-1}T_{1}=T_{1}S_{1}S_{1}^{\ast}$ and
      $T_{1}S_{1}S_{1}^{\ast}T_{1}^{\ast}D^{-1}T_{2}=0 $,  which are equivalent to (ii).
\item $T_{1}T_{1}^{\ast}D^{-1}T_{1}=T_{1}$ and $ [S_{1}S_{1}^{\ast}, T_{1}^{\ast}D^{-1}T_{1}]=0$,
        which are equivalent to (iii).
\end{enumerate}
 Note that $ [S_{1}S_{1}^{\ast},T_{1}^{\ast}D^{-1}T_{1}]=0$, since
 $(T_{1}S_{1})^{\dagger}T_{1}S_{1} $ is selfadjoint. We show that
 $(1)\Rightarrow (3)\Rightarrow (2)\Rightarrow (1)$.

Suppose (1) holds. If we multiply $(T_1S_1)^{ \dag}T_1S_1=S_1^{-1}T_1^{*}D^{-1}T_1S_1$
by $T_1S_1$ from the left side, we obtain $T_1=T_1T_1^{*}D^{-1}T_1$. Furthermore,
$[S_1S_1^{*}, T_1^{*}D^{-1}T_1]=0$, i.e. (3) holds.

Suppose (3) holds. Obviously, $T_{1}S_{1}S_{1}^{\ast}T_{1}^{\ast}
D^{-1}T_{1}=T_{1}T_{1}^{\ast}D^{-1}T_{1}S_{1}S_{1}^{\ast}=T_{1}S_{1}S_{1}^{\ast}$,
that is, the first equality of (2) holds. According to the fact that
$(T_{1} T_{1}^{*}+T_2 T_{2}^{*})D^{-1}T_1=T_1$ and the assumption
$T_{1}T_{1}^{\ast}D^{-1}T_{1}=T_{1}$, we have $T_{2}^{\ast}D^{-1}T_{1}=0$. Consequently,
$$Ran(D^{-1}T_1) \subseteq Ker(T_2 T_{2}^{*})=Ker(T_{2}^{*}),$$
which yields $T_{2}^{\ast}D^{-1}T_{1}=0$. Therefore, $T_{1}^{\ast}D^{-1}T_{2}=0$
which establishes the second equality of (2).

In order to prove $(2) \Rightarrow (1)$, we multiply $ T_{1}S_{1}
S_{1}^{\ast}T_{1}^{\ast}D^{-1}T_{1}=T_{1}S_{1}
S_{1}^{\ast}$ by $(T_{1}S_{1})^{\dagger}$ from the left side. In view of
$[S_1S_1^{*}, T_1^{*}D^{-1}T_1]=0$, we find
\begin{eqnarray*}
S_{1}^{\ast}T_{1}^{\ast}D^{-1}T_{1}=
(T_{1}S_{1})^{\dagger}T_{1}S_{1}S_{1}^{\ast} & \Rightarrow &
(T_{1}S_{1})^{\dagger}
T_{1}S_{1}=S_{1}^{\ast}T_{1}^{\ast}D^{-1}T_{1}(S_{1}^{\ast})^{-1} \\
& \Leftrightarrow & (T_{1}S_{1})^{\dagger}T_{1}S_{1}=S_{1}^{-1}
T_{1}^{\ast}D^{-1}T_{1}S_{1}.
\end{eqnarray*}
\end{proof}
Now we are ready to derive Greville's result, which also gives an answer to a
problem of \cite{SHA/PRODUCT} about the reverse order law for Moore-Penrose
inverses of modular operators.
The operators $SS^{\dagger}$ and $T^{\dagger}T$ are orthogonal projections onto
$Ran(S)$ and $Ran(T^{ \dag})=Ran(T^*)$, respectively. These facts together
with Theorems \ref{ahmadi3} and \ref{ahmadi4} lead us
to the following result.

\begin{corollary}\label{ahmadi5}
Suppose $E, F$ and $G$ are Hilbert C*-modules and $ S\in
\mathcal{L}(E,F) $, $T \in \mathcal{L}(F,G)$ and $TS \in \mathcal{L}(E,G)$
have closed ranges.
Then following statements are equivalent:
\begin{list}{(\roman{cou001})}{\usecounter{cou001}}
\item $(TS)^{\dagger}=S^{ \dag} T^{ \dag},$
\item $ TS(TS)^{\dagger}=TSS^{\dagger}T^{\dagger}$ and $(TS)^{\dagger}TS=S^{\dagger}T^{\dagger}TS,$
\item $SS^{\dagger}T^{\ast}TS=T^{\ast}TS$ and $TSS^{\ast}T^{\dagger}T=TSS^{\ast},$
\item $Ran(T^{\ast}TS) \subseteq Ran(S)$ and $Ran(SS^{\ast}T^{\ast}) \subseteq Ran(T^{\ast}).$
\end{list}
\end{corollary}

{\bf Acknowledgement}: The author is grateful to the referee for
his/her careful reading and his/her useful comments.


\end{document}